\documentclass{amsart}

\usepackage{charter}
\usepackage[a4paper,top=46mm,bottom=46mm,left=37mm,right=37mm]{geometry}

\usepackage{latexsym,amsfonts} 
\usepackage{amsmath,amssymb,graphics,setspace}
\usepackage{mathtools}
\usepackage{amsthm}
\usepackage{mathrsfs} 
\usepackage{tikz-cd} 
\usepackage{graphicx}
\usepackage{fancyhdr,txfonts,pxfonts}
\usepackage{paralist}
\usepackage{environ}
\usepackage{marvosym}
\usepackage[all,cmtip]{xy}

\usepackage{subfigure}
\renewcommand{\thesubfigure}{\thefigure.\arabic{subfigure}}
\makeatletter
\renewcommand{\p@subfigure}{}
\renewcommand{\@thesubfigure}{\thesubfigure:\hskip\subfiglabelskip}
\makeatother

\usepackage{pstricks,pst-text,pst-grad,pst-node,pst-3dplot,pstricks-add,pst-poly,pst-coil} 
\usepackage{pst-fun,pst-blur} 

\usepackage{hyperref}
\hypersetup{linktocpage=true,colorlinks=true,linkcolor=blue,citecolor=blue,pdfstartview={XYZ 1000 1000 1}}

\newcommand{\sd}{\mbox{SD}}
\newcommand{\tc}{\mbox{TC}}
\newcommand{\dis}{\mbox{D}}
\newcommand{\scat}{\mbox{scat}}
\newcommand{\cat}{\mbox{cat}}
\newcommand{\id}{\mbox{id}}

\newcommand{\ssd}{\mbox{sd}}

\newtheorem{example}{Example}
\newtheorem{remark}{Remark}
\newtheorem{definition}{Definition}

\newtheorem{theorem}{Theorem}
\newtheorem{proposition}{Proposition}
\newtheorem{corollary}{Corollary}

\begin{document} 
\title{Contiguity distance between simplicial maps}

\author[Ayse Borat]{Ayse Borat}
\address{
	Department of Mathematics, Bursa Technical University, Bursa, Turkey}
\email{ayse.borat@btu.edu.tr}

\author[Mehmetcik Pamuk]{Mehmetcik Pamuk}
\address{Department of Mathematics, Middle East Technical University, Ankara, Turkey}
\email{mpamuk@metu.edu.tr}

\author[Tane Vergili]{Tane Vergili}
\address{Department of Mathematics, Karadeniz Technical University, Trabzon, Turkey}
\email{tane.vergili@ktu.edu.tr}


\begin{abstract}
We study properties of contiguity distance between simplicial maps. In particular, we show that simplicial versions of $LS$-category and topological complexity are particular cases of this more general notion. 
\end{abstract}


\date{\today}

\maketitle
\tableofcontents

\section{Introduction}

The Lusternik-Schnirelmann category, introduced by Lusternik and Schnirelmann \cite{LS}, is an important numerical invariant concerning the critical points of smooth functions on manifolds.


\begin{definition}\cite{CLOT, LS} Lusternik Schnirelmann category of a space $X$, denoted by $\cat(X)$, is the least non-negative integer $k$ if there are open subsets $U_0, U_1, \ldots, U_k$ which cover $X$ such that each inclusion map $\iota_i:U_i\hookrightarrow X$ is nullhomotopic in $X$ for $i=0,1,\ldots,k$.
\end{definition}

Topological complexity of a topological space introduced by Farber \cite{F} is another numerical invariant closely related to motion planning problems.

\begin{definition}\cite{F} Let $\pi: PX\rightarrow X\times X$ be the path fibration. Topological complexity of a space $X$, denoted by $\tc(X)$, is the least non-negative integer $k$ if there are open subsets $U_0, U_1, \ldots, U_k$ which cover $X\times X$ such that on each $U_i$ there exists a continuous section of $\pi$ for $i=0,1,\ldots, k$.
\end{definition}

Although these invariants seem independent, they are similar in nature both being homotopy invariants. Macias-Virgos and Mosquera-Lois  \cite{VM:2019}  introduced homotopic distance, a notion   generalizing both $\cat$ and $\tc$. One of the advantages of homotopic distance is that since it is a number related to functions rather than spaces as in $\cat$ and $\tc$, we have the opportunity to investigate the behaviour of the homotopic distance under compositions which is not possible to do with $\cat$ and $\tc$. This feature leads us to prove the known $\tc$- and $\cat$-related theorems in an easy way.

\begin{definition}\cite{VM:2019} Let $f,g: X\rightarrow Y$ be continuous maps. Homotopic distance between $f$ and $g$, denoted by $\dis(f,g)$, is the least non-negative integer $k$ if there are open subsets $U_0, U_1, \ldots, U_k$ which cover $X$ such that $f|_{U_i}\simeq g|_{U_i}$ for all $i=0,1,\ldots,k$.
\end{definition}

In this paper, we consider the combinatorial objects, the simplicial complexes, and study the distance  between two simplicial maps adapted from homotopic distance. In fact one can consider a geometric realization of a  simplicial complex and study the ordinary homotopic distance between continuos maps induced by the geometric realization of simplicial maps. However, we opt to stay in the simplicial category in order not to loose the combinatorial aspects.  To do this, we consider a simplicial analogue of homotopic distance  between simplicial maps which relies on the contiguity. Then the simplicial analogues of $\cat$ and $\tc$ of a simplicial complex can be defined in terms of this distance. However we want to remark that the contiguity distance between simplicial maps and the homotopic distance between their corresponding geometric realizations might differ, see Example~\ref{ex:strict}. \\

Given a set $V$, an abstract simplicial complex with a vertex set $V$ is a set $K$ of finite subsets of $V$ such that the elements of $V$ belongs to $K$ and for any $\sigma \in K$ any subset of $\sigma$ belongs to $K$. The elements of $K$ are called the faces or the simplices of $K$. The dimension of an abstract  
simplex is just its cardinality minus 1 and the dimension of $K$ is the largest dimension of its simplices. For further details on abstract simplicial complexes, we refer to \cite{Kozlov:2008, S}.\\

The combinatorial description of any geometric simplicial complex $\tilde{K}$ obviously gives rise to an abstract simplicial complex $K$. One can always associate a geometric simplicial complex $\tilde{K}$ to an abstract simplicial complex $K$ in such a way that the combinatorial description of $\tilde{K}$ is the same as $K$ so that the underlying space of $\tilde{K}$ is homeomorphic to the geometric realization $|K|$. As a consequence, abstract simplicial complexes can be seen as topological spaces and geometric complexes can be seen as geometric realizations of their underlying combinatorial structure. So, one can consider simplicial complexes at the same time as combinatorial objects that are well-suited for effective computations and as topological spaces from which topological properties can be inferred.\\

It is a classical result that an arbitrary continuous map between geometric realizations of simplicial complexes can be deformed (after sufficiently many subdivisions) 
to a simplicial map, known as the simplicial approximation theorem.  But in general, simplicial approximations to a given continuous map are not unique.  An analogue 
of homotopy, called contiguity, is defined for simplicial maps so that different simplicial approximations to the same continuous map are contiguous.\\

\begin{definition}
Let $\varphi, \psi \colon K \to K'$ be two simplicial maps between simplicial complexes. We say that $\varphi$ and $\psi$ are \emph{contiguous}, denoted $\varphi \sim_c \psi$, 
provided  for a simplex $\sigma=\{v_0,\dotsc,v_n\}$ in $K$, the set of vertices $\varphi(\sigma) \cup \psi(\sigma)=\{ \varphi(v_0), \dotsc, \varphi(v_n), \psi(v_0), \dotsc, \psi(v_n) \}$ 
constitutes a simplex in $K'$.\\
\end{definition} 

For simplicial complexes and simplicial maps, the notion of contiguity can be considered as the discrete version of homotopy.
Being contiguous is a combinatorial condition which defines a reflexive and symmetric relation among simplicial maps.  On the other hand, this relation is not transitive.  
There is however an equivalence relation in the set of simplicial maps and the corresponding equivalence classes are called contiguity classes.\\

\begin{definition}
We say that two simplicial maps $\varphi, \psi \colon K\to K'$ are in the same contiguity class, denoted by $\varphi \sim \psi$, provided there exists a finite sequence of simplicial maps 
$\varphi_i \colon K\to K'$, $i=1,\dotsc,m$, such that $\varphi=\varphi_1\sim_c \varphi_2 \sim_c \cdots \sim_c \varphi_m=\psi$.  \\ 
\end{definition}


 Barmak and Minian \cite{Bar:2011, barmak-minian}  introduced the notion of strong collapse, a particular type of collapse which is specially adapted to the simplicial structure. Actually, it can be modelled as a simplicial map,
 in contrast with the standard concept of collapse which is not a simplicial map in general:  For a simplicial complex $K$, suppose that there is a pair of simplices $\sigma < \tau$ in $K$ such that $\sigma$ is a face of $\tau$ and $\sigma$ has no other cofaces. Such a simplex $\sigma$   is called a free face of $\tau$. Then the simplicial complex $K - \{\sigma, \tau\}$ is complex called an elementary collapse of $K$ (see Figure~\ref{fig:domination}). The action of collapsing is denoted $K \searrow K - \{\sigma, \tau\}$. The inverse of an elementary collapse is called an elementary expansion.  \\

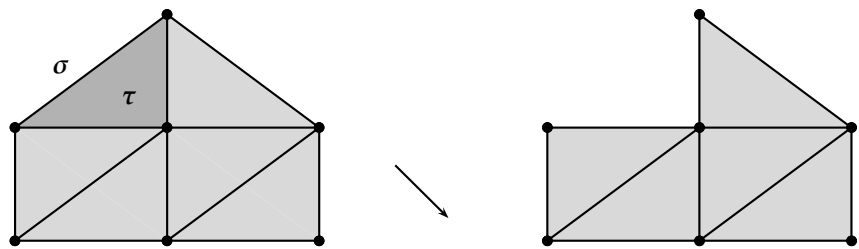
\begin{figure}[!ht]
	\centering
	\begin{pspicture}
	(-5.8,-2.7)(5.8,0.7)
	\psline*[linestyle=solid,linecolor=gray!30]%
	(-5.5,-1)(-3.5,-2.5)(-1.5,-2.5)(-3.5,-1)(-5.5,-1)(-5.5,-1)
	\psline*[linestyle=solid,linecolor=gray!30]%
	(-5.5,-1)(-5.5,-2.5)(-3.5,-2.5)(-5.5,-1) 
	\psline*[linestyle=solid,linecolor=gray!30]%
	(-3.5,-1)(-1.5,-2.5)(-1.5,-1)(-3.5,-1)
	\psline*[linestyle=solid,linecolor=gray!30]%
	(-1.5,-1)(-3.5,0.5)(-3.5,-1)(-1.5,-1)
	\psline*[linestyle=solid,linecolor=gray!60]%
	(-5.5,-1)(-3.5,-1)(-3.5,0.5)(-5.5,-1)
	\psline[linestyle=solid,linecolor=black]%
	(-5.5,-2.5)(-3.5,-2.5)(-1.5,-2.5)(-1.5,-1)(-3.5,0.5)(-5.5,-1)(-3.5,-1)(-1.5,-1)(-3.5,-2.5)(-3.5,-1)
	\psline[linestyle=solid,linecolor=black]%
	(-5.5,-1)(-5.5,-2.5)(-3.5,-1)(-3.5,0.5)
	\psdots[dotstyle=o,dotsize=4pt,linewidth=4pt,linecolor=black,fillcolor=black!100]%
	(-5.5,-2.5)(-3.5,-2.5)(-1.5,-2.5)(-5.5,-1)(-3.5,-1)(-1.5,-1)(-3.5,0.5)
	
	\rput(-4.9, -0.2){$\boldsymbol{\sigma}$}
	\rput(-4, -0.6){$\boldsymbol{\tau}$}
	
	\psline*[linestyle=solid,linecolor=gray!30]%
	(1.5,-2.5)(3.5,-1)(1.5,-1)(1.5,-2.5)
	\psline*[linestyle=solid,linecolor=gray!30]%
	(1.5,-2.5)(3.5,-1)(3.5,-2.5)(5.5,-1)(3.5,-1)
	\psline*[linestyle=solid,linecolor=gray!30]%
	(1.5,-2.5)(3.5,-2.5)(3.5,-1)(1.5,-2.5)
	\psline*[linestyle=solid,linecolor=gray!30]%
	(3.5,-2.5)(5.5,-2.5)(5.5,-1)(3.5,-2.5)
	\psline*[linestyle=solid,linecolor=gray!30]%
	(3.5,-1)(5.5,-1)(3.5,0.5) (3.5,-1)
	\psline[linestyle=solid,linecolor=black]%
	(1.5,-2.5)(3.5,-2.5)(5.5,-2.5)(5.5,-1)(3.5,0.5)(3.5,-1)(1.5,-1)(1.5,-2.5)(3.5,-1)(5.5,-1)(3.5,-2.5)(3.5,-1)
	\psdots[dotstyle=o,dotsize=4pt,linewidth=4pt,linecolor=black,fillcolor=black!100]%
	(1.5,-2.5)(3.5,-2.5)(5.5,-2.5)(1.5,-1)(3.5,-1)(5.5,-1)(3.5,0.5)
	
	\psline[arrows=->](-0.5,-1.5)(0.2,-2.2)	
	
	\end{pspicture}
	\caption[]{An elementary collapse}
	\label{fig:domination}
\end{figure}

A vertex $v$ of a simplicial complex $K$ is dominated by another vertex $v'$ if every maximal simplex that contains $v$ also contains $v'$. An elementary strong collapse consists of removing the open star of a dominated vertex $v$ from a simplicial complex $K$. The inverse of an elementary strong collapse is called an elementary strong expansion. A finite sequence of elementary strong collapses (expansions) is called a strong collapse (expansion), see Figure~\ref{fig:strongcollapse}. \\

\begin{figure}[!ht]
	\centering
	\begin{pspicture}
	(-5.8,-2.7)(5.8,0.7)
	\psline*[linestyle=solid,linecolor=gray!30]%
	(-5.5,-1)(-3.5,-2.5)(-1.5,-2.5)(-3.5,-1)(-5.5,-1)(-5.5,-1)
	\psline*[linestyle=solid,linecolor=gray!30]%
	(-5.5,-1)(-5.5,-2.5)(-3.5,-2.5)(-5.5,-1) 
	\psline*[linestyle=solid,linecolor=gray!30]%
	(-3.5,-1)(-1.5,-2.5)(-1.5,-1)(-3.5,-1)
	\psline*[linestyle=solid,linecolor=gray!70]%
	(-1.5,-1)(-3.5,0.5)(-3.5,-1)(-1.5,-1)
	\psline*[linestyle=solid,linecolor=gray!70]%
	(-5.5,-1)(-3.5,-1)(-3.5,0.5)(-5.5,-1)
	\psline[linestyle=solid,linecolor=black]%
	(-5.5,-2.5)(-3.5,-2.5)(-1.5,-2.5)(-1.5,-1)(-3.5,0.5)(-5.5,-1)(-3.5,-1)(-1.5,-1)(-3.5,-2.5)(-3.5,-1)
	\psline[linestyle=solid,linecolor=black]%
	(-5.5,-1)(-5.5,-2.5)(-3.5,-1)(-3.5,0.5)
	\psdots[dotstyle=o,dotsize=4pt,linewidth=4pt,linecolor=black,fillcolor=black!100]%
	(-5.5,-2.5)(-3.5,-2.5)(-1.5,-2.5)(-5.5,-1)(-3.5,-1)(-1.5,-1)(-3.5,0.5)
	
	\rput(-3.25, -0.75){$\boldsymbol{v'}$}
	\rput(-3.3, 0.6){$\boldsymbol{v}$}

	\psline*[linestyle=solid,linecolor=gray!30]%
	(1.5,-2.5)(3.5,-1)(1.5,-1)(1.5,-2.5)
	\psline*[linestyle=solid,linecolor=gray!30]%
	(1.5,-2.5)(3.5,-1)(3.5,-2.5)(5.5,-1)(3.5,-1)
	\psline*[linestyle=solid,linecolor=gray!30]%
	(1.5,-2.5)(3.5,-2.5)(3.5,-1)(1.5,-2.5)
	\psline*[linestyle=solid,linecolor=gray!30]%
	(3.5,-2.5)(5.5,-2.5)(5.5,-1)(3.5,-2.5)
	\psline[linestyle=solid,linecolor=black]%
	(1.5,-2.5)(3.5,-2.5)(5.5,-2.5)(5.5,-1)
	\psline[linestyle=solid,linecolor=black](3.5,-1)(1.5,-1)(1.5,-2.5)(3.5,-1)(5.5,-1)(3.5,-2.5)(3.5,-1)
	\psdots[dotstyle=o,dotsize=4pt,linewidth=4pt,linecolor=black,fillcolor=black!100]%
	(1.5,-2.5)(3.5,-2.5)(5.5,-2.5)(1.5,-1)(3.5,-1)(5.5,-1)
	
	\psline[arrows=->](-0.5,-1.5)(0.2,-2.2)	
	\psline[arrows=->] (-0.3,-1.5)(0.4,-2.2)
	\rput(3.7, -0.7){$\boldsymbol{v'}$}
	\end{pspicture}
	\caption[]{An elementary strong collapse}
	\label{fig:strongcollapse}
\end{figure}
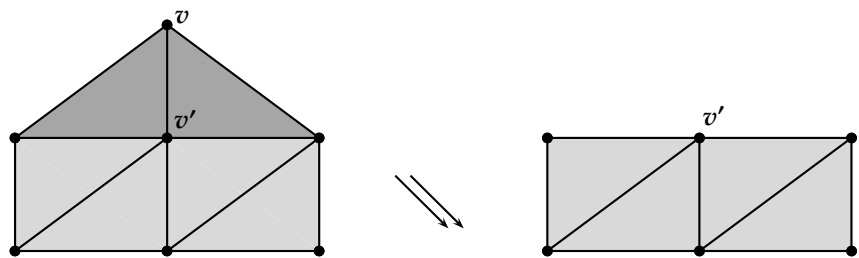

Two simplicial complexes $K$, $K'$ have the same strong homotopy type, denoted by $K \sim K'$, if they are related by a sequence of strong collapses and expansions. 
Surprisingly, this turns out to be intimately related to the classical notion of contiguity. More precisely, having the same strong homotopy type is equivalent to the existence of a strong equivalence. \\

 A simplicial map $\varphi \colon K\to K'$ is called a strong equivalence if there exists $\psi \colon K' \to K$ such that $\varphi \circ \psi \sim \id_{K'} $ and $\psi \circ \varphi \sim \id_K$.   The theory of strong homotopy types of simplicial complexes was introduced in \cite{barmak-minian}.  
 Strong homotopy types can be described by elementary moves called strong collapses.  From this theory, Birman and Minian obtained new results for studying simplicial collapsibility. \\

A natural definition of Lusternik-Schnirelman (LS) category for simplicial complexes, that is invariant under strong equivalences, is given in \cite{TerneroVirgosVilches:2015}
and a notion of discrete topological complexity in the setting of simplicial complexes by means of contiguous simplicial maps is given in \cite{TVMV:2018}.  \\

Let $K$ be a simplicial complex and $L\subseteq K$ a subcomplex. We say that $L$ is \emph{categorical}, provided there exists a vertex $v_0\in K$ such that the inclusion map  $i \colon L \hookrightarrow K$ and the constant map $c_{v_0}\colon L \to K$ are in the same contiguity class. The simplicial LS category, denoted \textbf{$\scat(K)$},  is defined as the least integer $n\geq 0$ such that $K$ is covered by $(n+1)$ categorical subcomplexes \cite{TerneroVirgosVilches:2015}.  Immediately from this definition, we can 
conclude that  a simplicial complex $K$ is strongly collapsible i.e., having the strong homotopy type of a point if and only if $\scat(K)=0$.\\


To define a simplicial structure for the Cartesian product of two simplicial complexes we need to give an  order to the vertices but still it does not satisfy the universal property of a product. As in \cite{Kozlov:2008},  we can define a product of  two simplicial complexes, called categorical product,  in such a way that their product is a  simplicial complex and satisfy the universal property of a product. Let $K$ and $K'$ be two  simplicial complexes. Then  the categorical product of $K$ and $K'$, denoted by $K\prod K'$ is an  simplicial complex  such that \\
	\begin{compactenum}
		\item its vertices are pairs $(v,\omega)$ where $v$ is  a vertex of $K$ and $\omega$ is a vertex of $K'$, and \\
		
		\item the projections $\mathrm{pr}_1: K \prod K'\to K$ and  $\mathrm{pr}_2: K \prod K'\to K'$  are simplicial maps and are universal with the property.\\
		
\end{compactenum}


 Let $K$ be a simplicial complex and $K^2=K\prod K$ a categorical product. Then a simplicial subcomplex $\Omega \subset K^2$ is a Farber subcomplex, provided there exists a simplicial map $\sigma\colon \Omega \to K$ such that $\Delta \circ \sigma \sim \iota_\Omega$ where $\iota_\Omega\colon \Omega \xhookrightarrow{} K^2$ is the inclusion map and $\Delta\colon K\to K^2$ is the diagonal map $\Delta(v)=(v,v)$.  \\

 \begin{definition} \cite{TVMV:2018}
 The discrete topological complexity $\tc(K)$ of the simplicial complex $K$ is the least integer $n\geq 0$ such that $K^2$ can be covered by $(n+1)$ Farber subcomplexes.\\
\end{definition}

In other words, $\tc(K)\leq n $ if and only if $K^2=\Omega_0 \cup \ldots \cup \Omega_n$, and there exist simplicial maps $\sigma_j\colon \Omega_j \to K$ such that $\Delta \circ \sigma_j \sim \iota_j$ where $\iota_j\colon \Omega_j \xhookrightarrow{}  K^2$ are inclusions for $j=0,\ldots,n$.\\

Before the end of this section, we remark that for a given simplicial complex $K$, $\cat(|K|)$ and $\tc(|K|)$ are lower bounds for $\scat(K)$ and $\tc (K)$, respectively.

\section{Contiguity Distance}
Throughout the paper, a simplicial complex is meant to be an abstract simplicial complex, all simplicial complexes are assumed to be  (edge-) path connected, and all maps between simplicial complexes are assumed to be simplicial maps.\\
 
\begin{definition} \cite[Definition~8.1]{VM:2019} 
For simplicial maps $\varphi,\psi \colon K \to K'$, the contiguity distance between $\varphi$ and $\psi$, denoted by $\sd(\varphi,\psi)$, 
is the least integer $n\geq 0$ such that there exists a covering of $K$ by subcomplexes $K_0, K_1, \ldots, K_n$ with the property that $\varphi{\big|}_{K_j}, \psi{\big|}_{K_j} \colon K_j \to K'$ 
are in the same contiguity class for all $j=0,1,\ldots, n$.\\
\end{definition}

\begin{remark} There is another simplicial version of homotopic distance, called simplicial distance,  introduced in \cite{B} and  given in the sense of Gonzalez \cite{JG}. The underlying idea in the simplicial distance $\operatorname{SimpD}$ lies in \cite[Lemma~1.1]{JG} even though $\operatorname{SimpD}$ has nothing to do with sections of a path fibration. It is not easy to compare simplicial distance and contiguity distance as it is not easy to compare simplicial complexity \cite{JG} and discrete topological complexity \cite{TVMV:2018}. 
\end{remark}

It is easy to see that the contiguity distance defines a symmetric relation on the set of simplicial maps and the contiguity distance between two maps is zero 
if and  only if they are in the same contiguity class.  The next proposition tells us that this notion is  well-defined on the set of equivalence classes of simplicial 
maps.\\


\begin{proposition} \label{prop:contiguity}
If $\varphi  \sim  \bar{\varphi}, \psi \sim \bar{\psi} \colon K \to K'$, then $\sd(\varphi,\psi)=\sd(\bar{\varphi},\bar{\psi})$. 
\end{proposition}

\begin{proof} 
Suppose first that $\sd(\varphi,\psi)=n$.  By definition this means that there exists a covering of $K$ by subcomplexes $K_0, K_1, \ldots, K_n$ with the property that  $\varphi{\big|}_{K_j}, \psi{\big|}_{K_j} \colon K \to K'$ are in the same contiguity class for all $j$.  Since $\varphi \sim \bar{\varphi}$ and $\psi \sim\bar{\psi}$, so their restrictions to $K_j$  are also in the same contiguity classes for all $j$.  Also recall that, contiguity classes are equivalence classes, so we have  $\bar{\varphi}{\big|}_{K_j}\sim \varphi{\big|}_{K_j}\sim\psi{\big|}_{K_j}\sim\bar{\psi}{\big|}_{K_j}$ for all $j$. Therefore $\sd(\bar{\varphi},\bar{\psi})\leq n$.  Starting with $\sd(\bar{\varphi}, \bar{\psi})$  gives us $\sd(\varphi,\psi) \leq \sd(\bar{\varphi},\bar{\psi})$, which finishes the proof.\\
\end{proof}

 We can use a finite covering of a complex $K$ to produce an upper bound for the simplicial distance between maps. \\

\begin{proposition} 
Given two simplicial maps $\varphi, \psi \colon   K \to K'$ and a finite covering of $K$ by subcomplexes $K_0, K_1, \ldots, K_n$,  we have
\[
\sd(\varphi,\psi)\leq \sum_{i=0}^n \sd(\varphi{\big|}_{K_j},\psi{\big|}_{K_j}) +n.
\] 
\end{proposition}

\begin{proof} 
Suppose $\sd(\varphi|_{K_j},\psi|_{K_j})=m_j$ for each $j=0,1,\dots, n$. So there exists a covering of $K_j$ by subcomplexes $K_j^0, K_j^1, \ldots, K_j^{m_j}$ such that $\varphi|_{K_j^i}\sim\psi|_{K_j^i}$. \\

\noindent The collection $\mathcal{K}=\big\{K_0^0,\ldots, K_0^{m_0},K_1^0,\ldots, K_1^{m_1},\ldots, K_n^0,\ldots, K_n^{m_n}\big\}$ is a covering for $K$ satisfying $\varphi|_{L}\sim\psi|_{L}$ for all $L \in \mathcal{K}$. So since the cardinality of $\mathcal{K}$ is $(m_0 + m_1 + \ldots + m_n)+n +1$, the required inequality holds. \\
\end{proof}

Next, we mention the relation between the simplicial $LS$-category and the contiguity distance between simplicial maps. First note that  for a subcomplex $L$ of a simplicial complex $K$, if $\id_K{\big|}_L$ and $c_{v_0}{\big|}_L$ are in the same contiguity class then $L$ is categorical in $K$. 
From this observation it is easy to see that for a simplicial complex $K$ and a vertex $v_0$ of $K$, we have 
\[
\scat(K)=\sd(\id_K, c_{v_0}).
\]

Let $K$ be a simplicial complex and $|K|$ denote its geometric realization. We know that both $\scat(K)$ and $\tc(K)$ might differ from $\cat (|K|)$ and $\tc(|K|)$ (see, Theorem~\ref{thm:catgeometric} and \cite[Theorem~5.2]{TVMV:2018}).  The simplicial category and discrete topological complexity depend on the simplicial structure more than on the geometric realization of the complex \cite{TVMV:2018, TerneroVirgosVilches:2015}. This implies that for  simplicial complexes $K, K'$ and simplicial maps $\varphi, \psi \colon K \to K'$, we expect $\sd(\varphi, \psi)$ is not necessarily the same as $\mathrm{D}(|\varphi|,|\psi|)$, where $|\varphi|, |\psi| \colon |K| \to |K'|$ are continuous maps between their corresponding geometric realizations \cite{VM:2019}. \\

\begin{proposition} \label{prop:distancerelation}
	For simplicial maps $\varphi, \psi: K\to L$, we have  $\mathrm{D}(|\varphi|,|\psi|) \leq \sd(\varphi, \psi)$. 
\end{proposition}
\begin{proof}  
	Let $\sd(\phi,\psi)=n$ so that there exist subcomplexes $K_0, K_1, \dotsc, K_n$ in such a way that the inclusion map $\iota_i\colon K_i \to K$ and the constant map  $c_v \colon K_i \to L$ are in the same contiguity class, $\iota_i \sim c_v$. Note that the union of the closed subsets  $|K_0|, |K_1|, \dotsc, |K_n|$ of $|K|$  covers $|K|$ and the geometric realizations of $\iota_i$ and $c_v$, 
	\[ |\iota_i|, |c_v| \colon |K_i| \to |K|\] 
	are homotopic continuous maps. 
\end{proof} 

The following is an example for the strict form of the inequality given in Proposition~\ref{prop:distancerelation}.

\begin{example}\label{ex:strict}
	Consider the simplicial complex $K$ given in Figure~\ref{fig:strict} \cite{barmak-minian}. Let $\id_K$ and $c$ be the identity simplicial map and constant simplicial map on $K$, respectively.  We know that $\scat K=1$ (\cite[Example~3.2]{TerneroVirgosVilches:2015}) so that $\sd(\id_{K}, c)=1$. Notice that the homotopic distance $\mathrm{D}(|\id_{K}|, |c|)$ is zero which follows from the fact that 
	the geometric realization $|K|$, of $K$ is contractible. Therefore $\mathrm{D}(|\id_{K}|, |c|) < \sd(\id_{K},c)$.
\end{example}

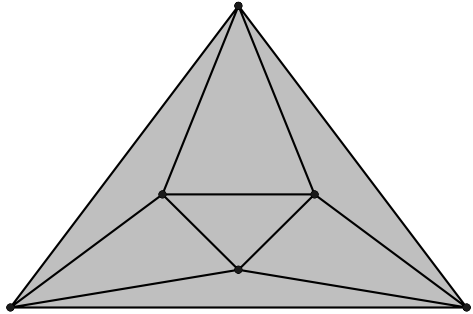
\begin{figure}[!ht]
	\centering
	\begin{pspicture}
	(-0.2,-0.1)(6.2,4.1)
	\psline*[linestyle=solid,linecolor=gray!50]%
	(0,0)(3,0.5)(6,0)(0,0)(3,4)(6,0)(4,1.5)(3,0.5)(2,1.5)(4,1.5)(3,4)(2,1.5)(0,0)
	\psline*[linestyle=solid,linecolor=gray!50]%
	(3,4)(2,1.5)(4,1.5)(3,4)
	\psline*[linestyle=solid,linecolor=gray!50]%
	(0,0)(2,1.5)(3,0.5)(0,0)
	\psline*[linestyle=solid,linecolor=gray!50]%
	(3,0.5)(4,1.5)(6,0)(3,0.5)
	\psline[linestyle=solid,linecolor=black]%
	(0,0)(3,0.5)(6,0)(0,0)(3,4)(6,0)(4,1.5)(3,0.5)(2,1.5)(4,1.5)(3,4)(2,1.5)(0,0)
	\psdots[dotstyle=o,dotsize=3pt,linewidth=3pt,linecolor=black,fillcolor=black!90]%
	(0,0)(6,0)(3,4)(3,0.5)(4,1.5)(2,1.5)
	
	\end{pspicture}
	\caption[]{$|K|$ is contractible whereas $K$ is not strongly collapsible}
	\label{fig:strict}
\end{figure}

Before we study the behaviour of contiguity distance under barycentric subdivision, we recall some ingredients and talk about how $\scat$ behaves under barycentric subdivision. \\

\begin{definition}
The barycentric subdivison of a given simplicial complex $K$ is the simplicial complex $\ssd K$ whose the set of vertices is $K$ and each $n$-simplex in $\ssd K$ is of form  $\{\sigma_0, \sigma_1,\ldots, \sigma_n\}$ where $\sigma_0 \subsetneq \sigma_1 \subsetneq \ldots  \subsetneq \sigma_n$. \\
\end{definition}

\begin{definition} For a simplicial map $\varphi \colon K \rightarrow L$, the induced map $\ssd \varphi \colon \ssd K \rightarrow \ssd L$ is given by 

\[
(\ssd \varphi)(\{\sigma_0, \ldots,\sigma_q\})=\{\varphi(\sigma_0),\ldots, \varphi(\sigma_q)\}.
\]

\end{definition}

\noindent Notice that $\ssd \varphi$ is a simplicial map and $\ssd (\id) = \id$.\\

\begin{proposition}\cite{TVMV:2019}\label{313} If the simplicial maps $\varphi, \psi \colon K \rightarrow L$ are in the same contiguity class, so are $\ssd \varphi$ and $\ssd \psi$.\\
\end{proposition}

The relation between the contiguity distance of two maps and the contiguity distance of their induced maps on barycentric subdivisions can be given as follows.

\begin{theorem}\label{thm:barycentric} For simplicial maps $\varphi,\psi\colon K \rightarrow K'$, $\sd(\ssd \varphi, \ssd \psi)\leq \sd(\varphi,\psi)$. 
\end{theorem}

\begin{proof} Let $\sd(\varphi,\psi)=n$. Then there are subcomplexes $K_0, K_1, \ldots, K_n$ covering $K$ such that $\varphi|_{K_i}\sim \psi|_{K_i}$ for all $i=0,1,\ldots, n$. \\

\noindent Take the cover $\{\ssd(K_0), \ssd(K_1), \ldots, \ssd(K_n)\}$ of $\ssd{K}$. By Proposition~\ref{313}, if $\varphi|_{K_i}\sim \psi|_{K_i}$ then $\ssd(\varphi|_{K_i})\sim \ssd(\psi|_{K_i})$.\\

\noindent On the other hand, $\ssd(\varphi|_{K_i})=\ssd\varphi|_{\ssd(K_i)}$. More precisely, if $\{\sigma_1, \ldots, \sigma_q\} \in \ssd(K_i)$,

\[
\ssd(\varphi|_{K_i})(\{\sigma_1, \ldots, \sigma_q\})=\big\{\ssd\varphi|_{K_i}(\sigma_1),\ldots, \ssd\varphi|_{K_i}(\sigma_q)\}=\{\ssd\varphi(\sigma_1), \ldots, \ssd\varphi(\sigma_q)\big\}=\ssd\varphi\big|_{\ssd(K_i)}. 
\]

\noindent Hence, since we have $\ssd(\varphi|_{K_i})\sim \ssd(\psi|_{K_i})$, it follows that $\ssd\varphi\big|_{\ssd(K_i)}\sim\ssd\psi\big|_{\ssd(K_i)}$ for all $i$.

\end{proof}

Although the below corollary is given as a consequence of some theorems related to finite spaces in \cite[Corollary 6.7]{TerneroVirgosVilches:2015} and a direct proof is given in \cite[Theorem 3.1.1]{TVMV:2019}, we give the following alternative proof using the contiguity distance for the consistency of the paper  \\

\begin{corollary} 
	For a simplicial complex $K$, $\scat(\ssd K)\leq \scat(K)$.
\end{corollary}

\begin{proof} In Theorem~\ref{thm:barycentric}, take $\varphi=\id$ and $\psi=c$ as the identity map and a constant map, respectively. So the induced maps $\ssd(\id)$ and $\ssd(c)$ are also the identity and constant map on $\ssd(K)$. Thus the corollary follows.

\end{proof}


	


Observe that, for a simplicial complex $K$ being strongly collapsible is equivalent to saying that $\scat(K)=\sd(\id_K, c_{v_0})=0$.  Hence, for a strongly collapsible complex $K$, we have  $\id_K\sim c_{v_0}$.  The following theorem tells that the same is true for arbitrary maps.\\

\begin{theorem}\label{thm:collapsible}
For any maps $\varphi, \psi \colon K \to K'$, $\sd(\varphi,\psi)=0$, provided   $K$ or $K'$ is strongly collapsible. 
\end{theorem}

\begin{proof}
Suppose $K$ is strongly collapsible, then we have    $\id_K \sim c_{v_0}$ where $v_0$ is a vertex in $K$. We have the following diagram 
\begin{equation*}
\xymatrix{ K \ar@/^/[r]^{id_K} \ar@/_/@{->}[r]_{c_{v_0}}   & K  \ar[r]^{\varphi}   & K'}
\end{equation*}
which implies that
\[
\varphi \circ  id_K \sim  \varphi \circ c_{v_0}  \quad \mbox{(constant)}.
\]

Similarly we have 
\begin{equation*}
\xymatrix{K \ar@/^/[r]^{id_K}\ar@/_/@{->}[r]_{c_{v_0}}   & K  \ar[r]^{\psi}   & K'}
\end{equation*}
so that 
\[
\psi \circ  id_K \sim  \psi \circ c_{v_0} \quad \mbox{(constant)}.
\]

Since $K'$ is edge-path connected,  all the constant maps are in the same contiguity class. Hence we have $\varphi= \varphi \circ  id_K \sim   \psi \circ  id_K = \psi$. \\
	
On the other hand, if $K'$ is strongly collapsible 
\[ \mbox{scat}(K')= 0 = SD(id_{K'},c_{\omega_0})\]
where $\omega_0$ is a vertex in $K'$. That is, 

\[ 
id_{K'} \sim c_{\omega_0}.
\]

This time we have the following diagram 
\begin{equation*}
\xymatrix{K   \ar[r]^{\varphi}  & K'   \ar@/^/[r]^{id_{K'}} \ar@/_/@{->}[r]_{c_{\omega_0}}  & K'}
\end{equation*}
so that
\[  id_{K'} \circ \varphi \sim c_{\omega_0} \circ \varphi \quad  \mbox{(constant)}. \]
	
Similarly we have 
\begin{equation*}
\xymatrix{K   \ar[r]^{\psi}  & K'   \ar@/^/[r]^{id_{K'}} \ar@/_/@{->}[r]_{c_{\omega_0}}  & K'}
\end{equation*}
so that
\[  
id_{K'} \circ \psi \sim c_{\omega_0} \circ \psi \quad  \mbox{(constant)}. 
\]

Note that $K'$ is edge-path connected since it is strongly collapsible. Hence we have $\varphi= id_{K'} \circ \varphi \sim id_{K'} \circ \psi = \psi$. \\
\end{proof}

For the converse we have the following result.

\begin{corollary}
Let $K$ be a simplicial complex. If $\sd(\varphi,\psi)=0$ for any pair of simplicial maps $\varphi, \psi\colon K\to K$, then $K$ is strongly collapsible. 
\end{corollary}
\begin{proof}
If we take  $\varphi=id_K$ and $\psi=c_{v_0}$ on a fixed vertex $v_0\in K$,  our assumption $SD(id_K,c_{v_0})=0$ implies that $\mathrm{scat}(K)=0$ which is equivalent to saying that $K$ is strongly collapsible. 
\end{proof}


\begin{theorem}\label{thm:inclusion}
Let $v_0$ be a vertex of the simplicial complex $K$.  For the simplicial maps 
\[
i_1, i_2\colon K \to K^2
\]
\noindent
defined as $i_1(\sigma)=(\sigma,v_0)$ and $i_2(\sigma)=(v_0, \sigma)$, we have
$\scat(K)=\sd(i_1, i_2)$.
\end{theorem}

 \begin{proof}
First we prove that $\sd(i_1,i_2)\leq \scat(K)$. Let $L\subseteq K$ be categorical. That is, there exists a vertex $v_0$ of $K$ such that the inclusion map $\iota\colon L \xhookrightarrow{}  K$ and the constant map $c_{v_0}\colon L \to K$ are in the same contiguity class. We want to show that $i_1{\big|}_L$ and $i_2{\big|}_L$ are also in the same contiguity class. Consider the following composition of simplicial maps 
\begin{equation*}
\xymatrix{L \ar[r]^{\Delta_L}  & L^2 \ar@/^/[r]^{\iota \prod c_{v_0}} \ar@/_/@{->}[r]_{c_{v_0} \prod \iota}  & K^2}
 \end{equation*}
 where $\Delta_L$ is the diagonal map of $L$, defined on the set of vertices by $v \to (v,v)$,  and $\iota \prod c_{v_0}$ and $c_{v_0} \prod \iota$ is the categorical product of $\iota$ and $c_{v_0}$. Then
\begin{displaymath}
i_1{\big|}_L= (\iota \prod c_{v_0})\circ \Delta_L,
\end{displaymath}
and 
\begin{displaymath}
i_2{\big|}_L= (c_{v_0} \prod \iota)\circ \Delta_L.
\end{displaymath}
  
Since $L$ is categorical,  then $\iota \sim c_{v_0}$.  We have 

\begin{align*}
\iota \prod c_{v_0} & \sim c_{v_0}\prod c_{v_0},\\
c_{v_0} \prod \iota &\sim c_{v_0}\prod c_{v_0}.
\end{align*}

\noindent This implies 
\[
\iota \prod c_{v_0}  \sim c_{v_0} \prod \iota 
\]
so that $(\iota \prod c_{v_0})\circ \Delta_L \sim (c_{v_0} \prod \iota)\circ \Delta_L$,  which proves our claim. \\

\noindent Next, we show that $\scat(K) \leq \sd(i_1,i_2) $. Assume that $L$ is a subcomplex of $K$ with $i_1{\big|}_L \sim i_2{\big|}_L$.  Let $p_i\colon K^2 \to K$ be the projection maps for $i=1,2$. Then $p_1\circ i_1{\big|}_L \sim p_1\circ i_2{\big|}_L$ so that $\iota\sim c_{v_0}$.\\
\end{proof}

The Proposition~\ref{prop:distancerelation}  leads to the following theorem.

\begin{theorem}\label{thm:catgeometric}
	Let $K$ be a simplicial complex and $|K|$  its geometric realization. $\cat(|K|) \leq \scat(K)$.
\end{theorem}
\begin{proof}  
	Consider the simplicial maps $i_1: K \to K^2$ and $i_2: K\to K^2$ defined in Theorem~\ref{thm:inclusion} so that  $\scat(K)=\sd(i_1, i_2)$. In that case, their geometric realizations 
	\[|i_1|, |i_2|\colon |K|\to |K^2|\] 
	are continuous maps. By Lemma~5.1 in \cite{TVMV:2018}, we know that $|K^2|$ and $|K|\times |K|$ are homotopy equivalent spaces. Let $u: |K^2|\to |K|\times |K|$ be the homotopy equivalence. Therefore the inclusion maps $\iota_1: |K|\to |K|\times |K|$ and $\iota_2: |K|\to |K|\times |K|$ are homotopic to $u\circ |i_1|$ and $u\circ |i_2|$ respectively. By Proposition~\ref{prop:distancerelation} and   \cite[Proposition~3.1]{VM:2019}, we have
	\begin{align*}
		\cat(|K|) = \mathrm{D}(\iota_1,\iota_2) = \mathrm{D}(u\circ |i_1|, u\circ |i_2|) \leq \mathrm{D}(|i_1|,|i_2|) \leq \sd(i_1, i_2) = \scat(K).
	\end{align*}
\end{proof}

Our next aim is to prove Theorem~\ref{invariant}. So we need Corollary~\ref{cor2} and Corollary~\ref{cor:homotopy} which follow from Proposition~\ref{prop4} and Proposition~\ref{prop:homotopy}, respectively.\\

\begin{proposition}\label{prop4}
Let $\varphi, \psi \colon K \to K'$ and $\mu \colon M \to K$ be simplicial maps. Then we have
\[
\sd(\varphi\circ \mu, \psi\circ \mu)\leq \sd(\varphi,\psi).
\]
\end{proposition}
\begin{proof} Let $\sd(\varphi,\psi)=n$. Then there exist subcomplexes $K_0, \ldots, K_n$ of $K$ such that $\varphi{\big|}_{K_j}\sim \psi{\big|}_{K_j}$ for all $j$. Define $M\supset M_j\colon=\mu^{-1}(K_j)$ and the restriction map $\mu_j\colon M_j \to K$. Then
\[
(\varphi\circ \mu)_j = \varphi \circ \mu_j = \varphi \circ \iota_j\circ \bar{\mu}_j = \varphi_j \circ \bar{\mu}_j \sim \psi_j \circ \bar{\mu}_j =  \psi \circ \iota_j\circ \bar{\mu}_j = \psi \circ \mu_j = (\psi\circ \mu)_j 
\]
where $\iota_j\colon K_j \xhookrightarrow{}  K$ is the inclusion and $\bar{\mu}_j: M_j \to K_j$, $\bar{\mu}_j(x)=\mu_j(x)$ is a map satisfying $\mu_j=\iota_j\circ \bar{\mu_j}$. Therefore $\sd(\varphi\circ\mu,\psi\circ\mu)\leq n$.\\
\end{proof}

\begin{corollary}\label{cor2}
Let $\varphi, \psi\colon K \to  K'$ be simplicial maps and $\beta\colon M \to K$ be a simplicial map which has a right strong equivalence (that is, $\beta$ satisfies $\beta \circ \alpha \sim \id_{K}$ where $\alpha\colon K \to M$). Then $\sd(\varphi\circ \beta,\psi\circ\beta)=\sd(\varphi, \psi)$.\\
\end{corollary}

\begin{proof}  Since $\beta \circ \alpha \sim \id_{K}$ it follows that $\varphi \circ \beta\circ \alpha \sim \varphi$ and $\psi \circ \beta\circ\alpha \sim \psi$. So 
\[
\sd(\varphi,\psi)=\sd(\varphi\circ \beta\circ\alpha,\psi \circ \beta\circ\alpha)\leq \sd(\varphi\circ \beta,\psi\circ \beta) \leq \sd(\varphi,\psi)
\]
where the equality follows from Proposition~\ref{prop:contiguity} and the inequalities follow from Proposition~\ref{prop4}. Hence we have $\sd(\varphi\circ \beta,\psi\circ\beta)=\sd(\varphi, \psi)$.\\

\end{proof}

\begin{proposition}\label{prop:homotopy} 
Let $\varphi, \psi\colon K \to K'$ and $\phi, \phi'\colon K' \to  M$ be simplicial maps. If $\phi\sim\phi'$, then $\sd(\phi \circ \varphi, \phi'\circ \psi)\leq \sd(\varphi,\psi)$. 
\end{proposition}
\begin{proof} Suppose $\sd(\varphi,\psi)=n$. Then there exist subcomplexes $K'_0, K'_1, \ldots, K'_n$ of $K'$ such that $\varphi{\big|}_{K'_i}$ and $\psi{\big|}_{K'_j}$ are in the same contiguity class for all $i,j$. So 
\[
(\phi \circ \varphi){\big|}_{K'_i}=\phi \circ \varphi{\big|}_{K'_i} \sim \phi' \circ \varphi{\big|}_{K'_i} \sim \phi' \circ \varphi{\big|}_{K'_j} = (\phi' \circ \varphi){\big|}_{K'_j}.
\]
Hence  $\sd(\phi \circ \varphi, \phi'\circ \psi)\leq n$.\\
\end{proof}

\begin{corollary}\label{cor:homotopy} 
Let $\varphi, \psi\colon K \to K'$ be simplicial maps and $\alpha\colon K' \to M$ be a simplicial map  which has a left strong equivalence (that is, $\alpha$ satisfies $\beta \circ \alpha \sim \id_{K'}$ where $\beta\colon M \to K'$). Then $\sd(\alpha\circ \varphi,\alpha\circ\psi)=\sd(\varphi, \psi)$.
\end{corollary}
\begin{proof} Since $\beta \circ \alpha \sim \id_{K'}$ it follows that $\beta\circ \alpha \circ\varphi \sim \varphi$ and $\beta\circ\alpha \circ \psi \sim \psi$. So 
\[
\sd(\varphi,\psi)=\sd(\beta\circ\alpha\circ \varphi,\beta\circ\alpha\circ\psi)\leq \sd(\alpha\circ \varphi,\alpha\circ \psi) \leq \sd(\varphi,\psi)
\]
where the equality follows from Proposition~\ref{prop:contiguity} and the inequalities follow from Proposition~\ref{prop:homotopy}. Hence we have $\sd(\alpha\circ\varphi,\alpha\circ\psi)=\sd(\varphi,\psi)$.\\
\end{proof}
\begin{theorem}\label{invariant} If $\beta: K' \sim K$ and $\alpha: L\sim L'$ have the same strong homotopy type and if simplicial maps $\varphi, \psi: K\rightarrow L$ and $\varphi',\psi': K'\rightarrow L'$ make the following diagrams commutative with respect to $f$ and $g$ respectively, in the sense of contiguity (that is, $\alpha\circ \varphi\circ \beta\sim \varphi'$ and $\alpha\circ\psi\circ\beta\sim\psi'$), then we have $\sd(\varphi,\psi)=\sd(\varphi',\psi')$.

\[
\begin{tikzcd}
K \arrow[r, "\varphi", shift left] \arrow[r, "\psi"', shift right]  
& L \arrow[d, "\alpha"] \\
K' \arrow[u, "\beta"] \arrow[r, "\varphi'", shift left] \arrow[r, "\psi'"', shift right] 
& L'
\end{tikzcd}
\]
\end{theorem}

\begin{proof} $\sd(\varphi',\psi')=\sd(\alpha\circ\varphi\circ\beta,\alpha\circ\psi\circ\beta)=\sd(\varphi\circ\beta,\psi\circ\beta)=\sd(\varphi,\psi)$\\

where the second equality follows from Corollary~\ref{cor:homotopy} and the last equality follows from Corollary~\ref{cor2}. 
\end{proof}

\begin{remark} Notice that the result of Theorem~\ref{invariant} is still valid even if we consider $\beta$ and $\alpha$ as right and left strong equivalences, respectively. \\
\end{remark}

The simplicial LS category of a simplicial map is defined as in the following definition.

\begin{definition}\cite{SS} 
Let $\varphi \colon K \to K'$ be a simplicial map and $\omega_0$ be a vertex of $K'$.  Simplicial LS category \textbf{$\scat(\varphi)$} of $\varphi$  is defined to be the least integer $n$  such that there exists a covering of $K$ by subcomplexes $K_0, K_1, \ldots, K_n$ such that  $\varphi{\big|}_{K_j}\colon K_j \to K'$ and  constant map $c_{\omega_0}\colon K_j \to K$ are in the same contiguity class for all $j$.\\
\end{definition}

\begin{corollary} Let $\varphi\colon K \to K'$ be a simplicial map. Then $\scat(\varphi)\leq \min\{ \scat(K), \scat(K')\}$.
\end{corollary}
\begin{proof} Let $\id_K\colon K \to K$ be the identity map and $c_{v_0}\colon K \to K$ be the constant map at the vertex $v_0$ in $K$. 
\[
\scat(\varphi)=\sd(\varphi, \varphi\circ c_{v_0})=\sd(\varphi \circ \id_K, \varphi\circ c_{v_0})\leq \sd(\id_K,c_{\varphi(v_0)})=\scat(K)
\]
where the inequality follows from Proposition~\ref{prop:homotopy}. Hence $\scat(\varphi)\leq \scat(K)$. \\

\noindent On the other hand, we have
\[
\scat(K')=\sd(\id_{K'}, c_{\omega_0})\geq \sd(\id_{K'}\circ \varphi, c_{\omega_0}\circ \varphi)=\scat(\varphi)
\]
where $\id_{K'}\colon K' \to K'$ is the identity map and $c_{\omega_0}\colon K' \to K'$ is the constant map at the vertex $\omega_0$ in $K'$. Thus $\scat(\varphi)\leq \scat(K')$.\\
\end{proof}

Let $K$ be a simplicial complex and $p_1,p_2\colon K^2 \to K$ projection maps onto the first and second factors, respectively. The following theorem is first proved in \cite[Theorem~3.4]{TVMV:2018} (see also \cite[Example~8.2]{VM:2019}). Here, we provide an alternative proof using contiguity distance. \\

\begin{theorem}\label{thm2}
	For a simplicial complex $K$, we have  $\sd(p_1,p_2)=\tc(K)$.
\end{theorem}
\begin{proof}
We first show that $\tc(K)\leq \sd(p_1,p_2)$. Suppose $\tc(K)=n$. Then there is a covering for $K^2$ which consists of Farber subcomplexes $L_0, L_1, \dotsc, L_n$. Since each $L_i$ is a Farber subcomplex, there exits a simplicial map $\sigma_i\colon L_i \to K$ such that $\Delta\circ \sigma_i \sim \iota_{L_i}$.  

\begin{align*}
\Delta\circ \sigma_i &\sim \iota_{L_i} \\
p_1\circ(\Delta\circ \sigma_i) &\sim p_1 \circ \iota_{L_i}=p_1{\big|}_{L_i} \\
p_2\circ(\Delta\circ \sigma_i) &\sim p_2 \circ \iota_{L_i} =p_2{\big|}_{L_i}\\
\end{align*}

Since $p_1\circ (\Delta\circ \sigma_i)=p_2\circ (\Delta\circ \sigma_i)$, we have $p_1{\big|}_{L_i}\sim p_2{\big|}_{L_i}$. \\
	
\noindent Next we will show that $\tc(K)\leq \sd(p_1,p_2)$. Suppose $\sd(p_1,p_2)=n$. Then there exist subcomplexes $L_0,L_1,\dotsc,L_n$ which cover $K^2$ and $p_1{\big|}_{L_i} \sim p_2{\big|}_{L_i}$ for $i=1,2,\dotsc,n$. WLOG, we assume $p_1{\big|}_{L_i} \sim_c p_2{\big|}_{L_i}$. This means that for an element $\big([x],[y]\big)$ in $L_i$ where  $[x]=\{x_1,x_2,\dotsc,x_k\}$ and $[y]=\{y_1,y_2,\dotsc,y_m\}$, 
\[
p_1\bigg(\big([x],[y]\big)\bigg) \cup p_2\bigg(\big([x],[y]\big)\bigg) = \{x_1,\dotsc, x_k, y_1, \dotsc,y_m\}
\] 
is a simplex in $K$.\\
	
\noindent We define a simplicial map $\sigma_i\colon L_i \to K$ so that
\begin{equation*}
\xymatrix{L_i \ar[r]^{\sigma}  & K \ar[r]^{\Delta}  & K^2}
\end{equation*}
$\Delta \circ \sigma_i \sim_c \iota_{L_i}$. \\
	
\noindent 
Define 
\[
\sigma_i \bigg(\big([x],[y]\big)\bigg) = p_1{\big|}_{L_i}\bigg(\big([x],[y]\big)\bigg) \cup  p_1{\big|}_{L_i}  \bigg(\big([x],[y]\big)\bigg) = \{x_1,\dotsc,x_k, y_1, \dotsc,y_m\} = 
\{x_1, \dotsc,x_k, y_1, \dotsc,y_m\}.
\]  

$\Delta \circ \sigma_i \bigg(\big([x],[y]\big)\bigg)= \big(\{x_1,\dotsc,x_k, y_1,\dotsc,y_m\}, \{x_1,\dotsc,x_k, y_1,\dotsc,y_m\}\big)$. \\
	 
$\iota_{L_i}\bigg(\big([x],[y]\big)\bigg)=\big([x],[y]\big) =\big(\{x_1,\dotsc,x_k\}, \{y_1,\dotsc,y_m\}\big)$. 
So, $L_i$ is also a Farber subcomplex.\\
\end{proof}

There is a well-known inequality between topological complexity and $LS$-category of  a  topological space $X$. The same inequality holds for simplicial complexes (see \cite[Theorem~4.3]{TVMV:2018}). In the following, we provide a proof in terms of contiguity distance.\\

\begin{theorem}
	For a simplicial complex $K$, we have $\mathrm{scat}(K)\leq TC(K)$.
\end{theorem}

\begin{proof}
	Consider the following composition of maps
	 \begin{equation*}
	\xymatrix{
		K  \ar[r]^{i_1}  & K^2   \ar[r]^{p_1}  &  K,}
	\end{equation*}
  \begin{equation*}
\xymatrix{
	v   \ar[r]^{i_1}  & (v, v_0)     \ar[r]^{p_1}  &  v,}
\end{equation*}
\noindent 
and note that $p_1 \circ i_1 = id_K$.
Similarly consider the composition of maps
	 \begin{equation*}
\xymatrix{
	K  \ar[r]^{i_1}  & K^2   \ar[r]^{p_2}  &  K,}
\end{equation*}
\begin{equation*}
\xymatrix{
v   \ar[r]^{i_1}  & (v, v_0)     \ar[r]^{p_2}  &  v_0,}
\end{equation*}
\noindent and we have $p_2 \circ i_1 = c_{v_0}$.  By Proposition~\ref{prop:homotopy}, 
\begin{align*}
& SD(p_1\circ i_1, p_2\circ i_2) \leq SD(p_1,p_2) =TC(K) \\
\Rightarrow  \quad &  SD(id_K, c_{v_0}) \leq TC(K) \\
\Rightarrow \quad  & \mathrm{scat}(K)\leq TC(K).
\end{align*}
\end{proof}

\begin{corollary} \label{cor:sdscat}
Let $\varphi, \psi\colon K \to K'$ be two simplicial maps (and $K'$ be edge path connected). Then $\sd(\varphi, \psi)\leq \scat(K)$. 
\end{corollary}
\begin{proof}
If we take $K''=K$, $\eta=\id_K$  and $\eta'=c_{v_0}$ a  constant map in Proposition~\ref{prop:composition}, then the constant maps $\varphi\circ c_{v_0}$ and $\psi\circ c_{v_0}\colon K\to K'$ are   in the same contiguity class since $K'$ is edge path connected. By Proposition~5 and Theorem~1 we have
 	\[ \sd(\varphi, \psi) = \sd(\varphi\circ \id_K, \psi\circ \id_K) \leq \sd (\id_K, c_{v_0}) =\scat(K).\]
\end{proof}

\begin{corollary}
	$\tc(K) \leq \scat(K^2)$. 
\end{corollary}
\begin{proof}
	If we consider the projection maps $p_1, p_2\colon K^2 \to K$ respectively in Corollary~\ref{cor:sdscat}, we have
	\[ \sd(p_1,p_2) = \tc(K) \leq \scat(K^2).  \]
\end{proof}

\begin{corollary}\label{cor:tc}
	Let $\varphi, \psi\colon K \to K'$ be two simplicial maps. Then $\sd(\varphi,\psi)\leq \tc(K')$.
\end{corollary}

\begin{proof}
Consider 
\begin{equation*}
\xymatrix{
K \prod K  \ar[r]^{\varphi \prod \psi}  & K' \prod K'   \ar@/^/[r]^{p_1}
\ar@/_/@{->}[r]_{p_2}  &  K'}
\end{equation*}
where each $p_i$  is a projection map for $i=1,2$. Then, using Proposition~\ref{prop:distancerelation}, we have 
\begin{align*}
\sd(\varphi, \psi)= \sd \big(p_1\circ(\varphi \prod \psi), p_2\circ(\varphi \prod \psi)\big) \leq \sd(p_1, p_2) = \tc(K').
\end{align*}
\end{proof}

\begin{remark}
	Observe that Theorem~\ref{thm:collapsible} also follows from Corollaries \ref{cor:sdscat} and \ref{cor:tc}.
\end{remark}

\begin{proposition} \label{prop:composition}
   Let $K, K'$, and $K''$ be simplicial complexes, $\eta,\eta'\colon K'' \to K$ and $\varphi,\psi\colon K \to K'$ be simplicial maps. If $\varphi \circ \eta' \sim \psi \circ \eta'$, then $\sd(\varphi\circ \eta, \psi \circ \eta) \leq \sd(\eta,\eta')$. 
\end{proposition}
\begin{proof}
	Let $\sd(\eta,\eta')=n$. Then there exists a covering $\{L_0,L_1,\dotsc, L_n\}$ for $K''$ such that $\eta{\big|}_{L_i} \sim \eta'{\big|}_{L_i}$ for $i=1,2,\dotsc,n$.
	\begin{align*}
	\eta{\big|}_{L_i} &\sim \eta'{\big|}_{L_i}\\
	\varphi \circ \eta{\big|}_{L_i} &\sim \varphi \circ \eta'{\big|}_{L_i}
	\end{align*}
	\begin{align*}
	\eta{\big|}_{L_i} &\sim \eta'{\big|}_{L_i}\\
	\psi \circ \eta{\big|}_{L_i} &\sim \psi \circ \eta'{\big|}_{L_i}
	\end{align*}
	Since $\varphi \circ \eta' \sim \psi \circ \eta'$, by the transitivity of $\sim$ we have  $\varphi \circ \eta{\big|}_{L_i} \sim 	\psi \circ \eta{\big|}_{L_i}$ and this completes our proof.\\
\end{proof}

\section{Acknowledgement} 

The authors  thank  Enrique Maci\'{a}s-Virg\'{o}s and  David Mosquera-Lois for their helpful comments and suggestions. The second author was partially supported by the Scientific and Technological Research Council of Turkey (T\"{U}B\.{I}TAK) [grant number 11F015].


\begin{thebibliography}{99}



    \bibitem{Bar:2011}   J.~A.~Barmak,
    \emph{Algebraic topology of finite topological spaces and applications}, Lecture notes in mathematics, Vol. 2032, Springer Heidelberg, 2011.\\
    
   

    \bibitem{barmak-minian}  J.~A.~Barmak, E.~G.~Minian, \emph{Strong homotopy types, nerves and collapses}, Discrete Comput. Geom. 47 (2) (2012) 301--328.\\
    
    \bibitem{B} A. Borat, Simplicial distance, submitted. arXiv: 2009.01640.\\   
    
     \bibitem{CLOT} O. Cornea, G. Lupton, J. Oprea, D. Tanre, Lusternik-Schnirelmann category. Mathematical Surveys and Monographs, Vol. 103, American Mathematical Society 2003.\\


    \bibitem{F} M. Farber, Topological complexity of motion planning, Discrete and Computational Geometry 29 (2003), 211-221.\\

    \bibitem{JG} J. Gonzalez, Simplicial complexity: piecewise linear motion planning in robotics, New York Journal of Mathematics, Vol 24 (2018), 279-292.\\

	\bibitem{VM:2019} E. Maci\'{a}s-Virg\'{o}s, D. Mosquera-Lois, Homotopic distance between maps, to appear in Math. Proc. Cambridge Philos. Soc. ArXiv: 1810.12591.\\
	
	\bibitem{TVMV:2018} D. Fern\'{a}ndez-Ternero, E.  Macias-Virg\'{o}s, E. Minuz, J. A. Vilches, Discrete topological complexity, Proc. Amer. Math. Soc. 146, 4535-4548 (2018).\\
	
	\bibitem{TVMV:2019} D. Fern\'{a}ndez-Ternero, E.  Macias-Virg\'{o}s, E. Minuz, J. A. Vilches, Simplicial Lusternik-Schnirelmann category, Publicacions Matematiques, Volume 63, Number 1,  265-293 (2019).\\	
	
	\bibitem{TerneroVirgosVilches:2015} D. Fern\'{a}ndez-Ternero, E.  Macias-Virg\'{o}s, J. A. Vilches, Lusternik-Schnirelmann category of simplicial complexes and finite spaces, Topology and its Applications, 194, 37-50 (2015).\\	

	
	\bibitem{FG} M. Farber, M. Grant, Symmetric motion planning. In M Farber, R Ghrist, M Burger, and D Koditschek editors, Topology and Robotics, Contemporary Mathematics Series, pages 85 - 104, United States, 2007. American Mathematical Society. \\

	\bibitem{FG2} M. Farber, M. Grant, Robot motion planning, weights of cohomology classes, and cohomology operations. Proc. Amer. Math. Soc., 136 (9), 3339 - 3349 (2008). \\


    \bibitem{Kozlov:2008} D. Kozlov, Combinatorial Algebraic Topology, Springer Science and Business Media, 2008.\\
    
    \bibitem{LS} L. Lusternik, L. Schnirelmann, Methodes Topoligiques dans les Problemes Variationnets. Herman, Paris (1934). \\

	\bibitem{SS} N. A. Scoville, W. Swei, On the Lusternik–Schnirelmann category of a simplicial map, Topology and Its Applications, 216, 116-128 (2017).\\
	
    \bibitem{S} E.H. Spanier, Algebraic Topology. McGraw-Hill Book Co., New York-Toronto, Ony.-London, 1966.\\
	

\end{thebibliography}
\end{document}